\documentclass[12pt,letterpaper]{amsart}

\usepackage{accents}
\usepackage{enumitem}
\usepackage{cite}

\usepackage{amsmath, nccmath}

\newtheorem{theorem}{Theorem}[section]
\newtheorem{lemma}[theorem]{Lemma}
\theoremstyle{definition}
\newtheorem{definition}[theorem]{Definition}

\newtheorem{corollary}[theorem]{Corollary}
\newtheorem{proposition}[theorem]{Proposition}
\theoremstyle{remark}

\newtheorem{claim}[theorem]{{\bf Claim}}

\numberwithin{equation}{section}

\usepackage{color}

\begin{document}

\def\C{\mathbb C}
\def\R{\mathbb R}
\def\X{\mathbb X}
\def\Z{\mathbb Z}
\def\Y{\mathbb Y}
\def\Z{\mathbb Z}
\def\N{\mathbb N}
\def\cal{\mathcal}
\def\cD{\cal D}
\def\tD{\tilde{{\cal D}}}
\def\F{\cal F}
\def\tf{\tilde{f}}
\def\tg{\tilde{g}}
\def\tu{\tilde{u}}

\def\cal{\mathcal}
\def\b{\mathcal B}
\def\c{\mathcal C}
\def\cc{\mathbb C}
\def\x{\mathbb X}
\def\r{\mathbb R}
\def\T{\mathbb T}
\def\uu{(U(t,s))_{t\ge s}}
\def\vv{(V(t,s))_{t\ge s}}
\def\xx{(X(t,s))_{t\ge s}}
\def\yy{(Y(t,s))_{t\ge s}}
\def\zz{(Z(t,s))_{t\ge s}}
\def\ss{(S(t))_{t\ge 0}}
\def\tt{(T(t,s))_{t\ge s}}
\def\rr{(R(t))_{t\ge 0}}
\title[Diffusive Delay in Lotka Volterra Models]{\textbf{EXISTENCE OF TRAVELING WAVES OF LOTKA VOLTERRA TYPE MODELS WITH DELAYED DIFFUSION TERM  AND PARTIAL QUASIMONOTONICITY  }}

\author[Barker]{William Barker}
\address{Department of Mathematics and Statistics, University of Arkansas at Little Rock, 2801 S University Ave, Little Rock, AR 72204. USA}
\email{wkbarker@ualr.edu}

\thanks{}

\date{\today}
\begin{abstract} This paper is concerned with the existence of traveling wave solutions for diffusive two-species Lotka-Volterra  systems with delay in both the reaction and diffusion terms without monotonicity. We extend the partial or cross monotone iteration method to systems that satisfy the partial quasi-monotone condition via construction appropriate upper and lower solutions . This is done by using Schauder’s fixed point theorem. 
\end{abstract}

\maketitle
\section{Introduction}
\noindent Nonlinear parabolic partial differential equations play a vital role in mathematical modeling for engineering as well as the physical and biological sciences. The specific type of parabolic equations that we are interested in are called reaction diffusion equations. When no delay is present there is a myriad of scholarly work. In particular, there are very elegant and classical results  for traveling wave front solutions of PDE, see  \cite{Dalecki,fif,Fisher,Kolmogorov,Murray1,Murray2,HSmith,volvolvol}.  

\medskip
\noindent 
The first work that studied the existence and properties of traveling waves in diffusion-reaction equations with delay in reaction term is attributed to Schaaf, \cite{sch}. The methods  used by Schaaf were phase space
analysis and the maximum principle for parabolic functional differential equations for a Fisher type nonlinearity. 

\medskip
\noindent
The seminal work for constructing monotone wave fronts for reaction diffusion is the paper by Wu and Zou \cite{wuzou}. Wu and Zou studied the existence of traveling wave solutions with a singular delay in the reaction term. The existence of such solutions was established under a quasimonotone or exponential quasimonotone property  via an iteration of appropriate upper and lower solutions. This is the so-called monotone iteration method. 

\medskip
\noindent
The above results were extended by Ma, \cite{ma} by employing the Schauder fixed point theorem using the decay norm. Moreover, Ma developed so-called super and sub solutions. This formulation relaxes the requirements of the upper and lower solutions in Wu and Zou.

\medskip
\noindent
Boumenir and Nguyen, \cite{boumin} introduced the concept of quasi-upper/lower solutions, which removes the $C^2$ requirement for the initial construction of upper and lower solutions. It was also shown that the quasi solutions become upper and lower solutions after one iteration. For an overview of results for traveling waves of reaction diffusion equations with delay, see \cite{Al-Omari,boumin,fanshashi,Gour,hale,Huang,liumeiyan,ma,sch,HSmith,vol,volvolvol,Wei2,wuzou,wuzouerr,Zhao}, their progeny and references therein. 

\medskip
\noindent
An important model in Biology is the so called Lotka-Volterra diffusion-cooperation system. Traveling waves solutions are important in such models. The literature for traveling diffusion-cooperation models have been concerned with either no delay or delay in the reaction term. 
Lotka-Volterra models of the form \begin{align}\label{LVRD}
\frac{\partial u (x,t)}{\partial t}&=D_1\frac{\partial^2u(x,t)}{\partial x^2}+\alpha_1 u(x,t)\left(1-a_1u(x,t-\tau_1)-b_1v(x,t-\tau_2)\right) \\
\frac{\partial v (x,t)}{\partial t}&=D_2\frac{\partial^2v(x,t)}{\partial x^2}+\alpha_2v(x,t)\left(1-b_2v(x,t-\tau_3)-a_2u(x,t-\tau_4)\right). \nonumber
\end{align}
where, $u(x,t), v(x,t)$  represent the population densities of two competing species, $D_1, D_2$ are positive diffusion constants, $\alpha_1,\alpha_2$ is the respective intrinsic population growth for $u, v,$ and $a_1, b_1,a_2,b_2$ are positive constants have been studied in \cite{Huazou2,lilin,lvwang,pan,Feng,Ruan,Wangzhou} and references therein. 

\medskip
\noindent
It is important to mention that the standard monotone iteration and relaxed exponential monotone iteration methods fail to allow for the construction of upper/lower solutions. A different iteration method is required, which can be referred to partial quasi-monotone, cross, or mixed interaction method.

 \medskip
\noindent
The existence of traveling waves for systems of the form (\ref{LVRD}) was studied in Li, et al. \cite{lilin} via constructing weak upper and lower solutions and using the Schauder fixed point theorem.  Feng et al. \cite{Feng, Ruan} obtained similar results for ratio dependent multi-species Lotka-Volterra competition models.    

\medskip
\noindent
The main purpose of this paper is to establish the existence of traveling wave solutions of the following Lotka-Volterra competition-cooperation model with delay in the diffusion term 
\begin{align}\label{LV1}
\frac{\partial u (x,t)}{\partial t}&=D_1\frac{\partial^2u(x,t-\tau_1)}{\partial x^2}+\alpha_1 u(x,t)\left(1-a_1u(x,t)-b_1v(x,t-\tau_2)\right) \\
\frac{\partial v (x,t)}{\partial t}&=D_2\frac{\partial^2v(x,t-\tau_3)}{\partial x^2}+\alpha_2v(x,t)\left(1-b_2v(x,t)-a_2u(x,t-\tau_4)\right). \nonumber
\end{align}
 Here, $u(x,t), v(x,t)$  represent the population densities of two competing species, $D_1, D_2$ are positive diffusion constants, $\alpha_1,\alpha_2$ is the respective intrinsic population growth for $u, v,.$ The constants $a_1, b_1,a_2,b_2$ are positive and the values $\tau_i, i=1,2,3,4 $ are positive time delays. 

 \medskip
\noindent
In order to use a monotone or partial monotone iteration method the positiveness of the unique bounded solutions needs to be established. The question of placing delay in the diffusion term has been difficult to answer. In Barker and Nguyen, \cite{BarkNguy} it was shown that the second order delay equation of the form $x''(t)-ax'(t+r)-bx(t+r)=f(t)$, where $a\not=0, b, r>0$ has a unique bounded solution for each given bounded and continuous $f(t)$. Moreover, if $r>0$ is sufficiently small and $f(t)\ge 0$ for $t\in\R$, then the unique bounded solution $x_f(t)\le 0$ for all $t\in\R$. This was done by using complex analysis to determine properties of the characteristic equation, for more information see \cite{cookegrossman,die,hale,mal,murnaimin,pru,Wei1}, and references therein. Our question seems to differ from the standard study where delay in the diffusion term is not considered, see\cite{Al-Omari,boumin,fanshashi,Gour,hale,Huang,Huazou2,lilin,liumeiyan,lvwang,ma,pan,sch,HSmith,Feng,Ruan,vol,volvolvol,Wei2,Wangzhou,wuzou,wuzouerr,Zhao}. 

\medskip
\noindent
The paper will be organized as follows: in Section 2 we will put forth preliminary information that will be utilized throughout the paper. In Section 3, we will show the existence of traveling waves for the system (\ref{LV1}) via Schauder's fixed point theorem in a Banach space equipped with exponential decay norm. This will be aided by partial monotone conditions that will be placed on appropriate upper and lower solutions. 

\medskip
\noindent
It will also be shown that the existence of quasi upper and lower solutions imply the existence of smooth upper and lower solutions.  This will also imply the existence of quasi upper and lower solutions is a sufficient condition for the existence of traveling waves.  In Section 4, we will construct quasi upper and lower solutions to an interesting Lotka-Volterra  competition model for two species. 

\section{Preliminaries}
\noindent 
In this paper we will use some standard notations as $\R^n,\C^n$ for the fields of reals and complex numbers in $n$ dimensions. We also take standard ordering for $\R^2$. This means $u=(u_1,u_2)^T$ and $v=(v_1,v_2)^T.$ We say that $u\le v$ if $u_1\le v_1, \ u_2\le v_2.$ We also say that $u<v$ if $u_1<v_1, \ u_2<v_2.$ We also take $|\cdot|$ to be the Euclidean norm in $\R^2.$ 

\medskip
\noindent
The space of all bounded and continuous functions from $\R  \to \R^n$ is denoted by $BC(\R,\R^n)$ which is equipped with the sup-norm $|| f|| := \sup_{t\in\R} \| f(t)\|,$ where $f(t)\in C\left(U, \R^2\right)$, where $U\subset \R$. $BC^k(\R,\R^n)$ stands for the space of all $k$-time continuously differentiable functions $\R\to\R^n$ such that all derivatives up to order $k$ are bounded.

\medskip
\noindent
If the boundedness is dropped from the above function spaces we will simply denote them by $C(\R,\R^n)$ and $C^k(\R,\R^n)$. For $f,g\in BC(\R,\R^n)$ we will use the natural order $f\le g$ if and only if $f(t)\le g(t)$ for all $t\in \R$. A constant function $f(t)=\alpha$ for all $t\in \R$ will be denoted by $\hat \alpha.$ 

\medskip
\noindent
Moreover, we note that $f_{ic}(\phi_{\theta},\psi_{\theta}): \X_{c\tau} = C\left([-\tau,0], \R^2\right)\to \R$ is defined by 
\[f_{ic}(\phi,\psi)= f_i\left(\phi^c,\psi^c\right), \ \phi^c(\theta)=\phi(c\theta), \ \psi^c(\theta)=\psi(c\theta), \ \theta \in [\tau,0], \ i=1,2.  \]
This allows us to see that the system (\ref{LV1}) under a wave front transformation can be written as
\begin{align}\label{LV3}
&D_1\phi''(t)-c\phi'(t+r_1)+f_{1c}(\phi_t, \psi_t)=0 \\
&D_2\psi''(t)-c\psi'(t+r_3)+f_{2c}(\phi_t, \psi_t)=0, \nonumber
\end{align}
where the delay has been moved out of the higher order term via an invariant translation. We also implement the following asymptotic boundary conditions
\begin{alignat*}{2}
&\lim_{t\to -\infty}\phi(t)=0,\ &&\lim_{t\to \infty}\phi(t)= k_1\\
&\lim_{t\to -\infty}\psi(t)=0,\ &&\lim_{t\to \infty}\psi(t)= k_2,
\end{alignat*}
where $0<k_1\le 1, \ 0<k_2\le 1.$ The following conditions for the reaction terms are employed
\begin{enumerate}[label=(C\arabic*)]
\item $f(\hat{0})=f(\hat{K})=0,$ where $K=(k_1,k_2)$ 
\item There are two Lipchitz constants $L_1>0, \ L_2>0$ such that 
\begin{align*}
 &\left|f_1(\phi_1,\psi_1)-f_1(\phi_2,\psi_2)\right|\le L_1 \left|\left|\Phi-\Psi\right|\right|  \\
 &\left|f_2(\phi_1,\psi_1)-f_2(\phi_2,\psi_2)\right|\le L_2 \left|\left|\Phi-\Psi\right|\right|,
\end{align*}
\end{enumerate}
for $\Phi =(\phi_1,\psi_1), \ \Psi =(\phi_2,\psi_2) \in C\left([\tau,0], \R\right)),$ and $0\le \Phi(s),\Psi(s)\le K.$

\section{The Case of Partial Quasimonotonicity}

\subsection{Partial Quasimonotonicity Conditions}
In this section, we develop sufficient conditions for the existence of traveling wave solutions for equations of the form (\ref{LV3}). We invoke the following cross iteration scheme.
\begin{definition}{Partial Quasi-Monotone Condition (PQM)}
Fix two constants $\beta_1, \ \beta_2>0$ such that 
\begin{enumerate}[label=(P\arabic*)]
\item $f_{1}\left(\phi_1,\psi_1\right)-f_{1}\left(\phi_2,\psi_1\right)+\beta_1\left[\phi_1(0)-\phi_2(0)\right]\ge 0,$
\item $f_{1}\left(\phi_1,\psi_1\right)-f_{1}\left(\phi_1,\psi_2\right)\le 0,$
\item $f_{2}\left(\phi_1,\psi_1\right)-f_{2}\left(\phi_2,\psi_2\right)+\beta_2\left[\psi_1(0)-\psi_2(0)\right]\ge 0.$
\item $f_{2}\left(\phi_1,\psi_1\right)-f_{2}\left(\phi_2,\psi_1\right)\le 0,$
\end{enumerate}
where $\phi_1,\ \phi_2, \ \psi_1,\ \psi_2 \in C(\R,\R),\ 0\le \phi_2\le \phi_1\le k_1,\ 0\le \psi_2\le \psi_1\le k_2.$ 
\end{definition}

\medskip
\noindent
Let $\beta_1, \ \beta_2$ be the constants in (PQM), define the following operators 
\begin{align}\label{H}
&H_1(\phi,\psi)(t)=f_{1c}\left(\phi_t,\psi_t\right)+\beta_1\phi(t+r_1), \ \phi, \psi \in C(\R,\R)  \\
&H_2(\phi,\psi)(t)=f_{2c}\left(\phi_t,\psi_t\right)+\beta_2\psi(t+r_3), \ \phi, \psi \in C(\R,\R). 
\end{align}
The operators $H_1, H_2$ satisfy the following properties:
\begin{lemma}\label{lemH}
Assume $(C1)$ and (PQM) hold, then 
\begin{enumerate}
    \item $H_1(\phi_2,\psi_1)(t)\le H_1(\phi_1,\psi_1)(t)$
    \item $H_1(\phi_1,\psi_1)(t)\le H_1(\phi_1,\psi_2)(t)$
    \item $H_2\left(\phi,\psi\right)(t)\ge 0$
    \item $H_2\left(\phi,\psi\right)(t)$ is non-decreasing for $t\in \R$
    \item $H_2\left(\phi_2,\psi_2\right)(t)\le H_2\left(\phi_1,\psi_1\right)(t).$
\end{enumerate}    
\end{lemma}
\begin{proof}
Looking towards $i.).$ By (PQM) it is clear
\begin{align*}
  &H_1(\phi_2,\psi_1)(t)- H_1(\phi_1,\psi_1)(t)=  f_{1c}\left(\phi_{2t},\psi_{1t}\right)+\beta_1\phi_2(t+r_1)-\left(f_{1c}\left(\phi_{1t},\psi_{1t}\right)+\beta_1\phi_1(t+r_1)\right)\\
 & =f_{1c}\left(\phi_{2t},\psi_{1t}\right)-f_{1c}\left(\phi_{1t},\psi_{1t}\right)+\beta_1\left(\phi_2(t+r_1)-\phi_1(t+r_1)\right)\\
 &=-\left[f_{1c}\left(\phi_{1t},\psi_{1t}\right)-f_{1c}\left(\phi_{2t},\psi_{1t}\right)+\beta_1\left(\phi_1(t+r_1)-\phi_2(t+r_1)\right)\right]\le 0,
\end{align*}
by (P1) pf (PQM). Thus, $H_1(\phi_2,\psi_1)(t)- H_1(\phi_1,\psi_1)(t)\le 0$. This finishes $i.)$ We will show $ii.)$ via direct computation as well. Indeed, 
\begin{align*}
&H_1(\phi_1,\psi_1)(t)- H_1(\phi_1,\psi_2)(t)=f_{1c}\left(\phi_{1t},\psi_{1t}\right)+\beta_1\phi_1(t+r_1)-\left(f_{1c}\left(\phi_{1t},\psi_{2t}\right)+\beta_1\phi_1(t+r_1)\right)\\
&= f_{1c}\left(\phi_{1t},\psi_{1t}\right)-f_{1c}\left(\phi_{1t},\psi_{2t}\right)\le 0
\end{align*}
by (P2) of (PQM). This proves $ii.)$ Part $iii.)$ follows directly from $(C_1)$ and each $\phi,\psi,\beta_2\ge 0.$ In order to prove $iv.)$ we fix $t\in \R, s>0,$ then 
\begin{align*}
& 0\le \phi_t\le \phi_{t+s}\le k_1\\
& 0\le \psi_t\le \psi_{t+s}\le k_2
\end{align*}
due to the fact that both $\phi$ and $\psi$ are non-decreasing. From here, we have
\begin{align*}
&H_2\left(\phi,\psi\right)(t+s)- H_2\left(\phi,\psi\right)(t)=  f_{2c}\left(\phi_{t+s},\psi_{t+s}\right)+\beta_2\phi(t+s)-\left(f_{2c}\left(\phi_{t},\psi_{t}\right)+\beta_2\phi(t)\right)\\
&=f_{2c}\left(\phi_{t+s},\psi_{t+s}\right)-f_{2c}\left(\phi_{t},\psi_{t}\right)+\beta_2\left[\phi(t+s)-\phi(t)\right]\\
&=f_{2c}\left(\phi_{t+s},\psi_{t+s}\right)-f_{2c}\left(\phi_{t},\psi_{t}\right)+\beta_2\left[\phi_{t+s}(0)-\phi_{t}(0)\right]\ge 0
\end{align*}
by (P3) of (PQM). This proves $iv.)$ Part $v.)$ follows directly from (P3) of (PQM). This completes the proof.
\end{proof}

\medskip
\noindent
We can write the system (\ref{LV3}) for all $t\in \R$ as
\begin{align}\label{LV4}
&D_1\phi''(t)-c\phi'(t+r_1)-\beta_1\phi(t+r_1) +H_1\left(\phi,\psi\right)(t)=0  \\
&D_2\psi''(t)-c\psi'(t+r_3)-\beta_2\psi(t+r_3)+H_2\left(\phi,\psi\right)(t)=0. \nonumber
\end{align}

\subsection{Existence of Traveling Wave Solutions} 
In this section we will consider existence of positive solutions for the system (\ref{LV4}). In \cite{BarkNguy}, it was shown that equations of the form $x''(t)-ax'(t+r)-bx(t+r)=f(t)$ have positive solutions where $r\in \R, \ f\in BC\left(\R, \R\right).$ It was also noted that these results can be extended to higher dimensions. 

\medskip
\noindent
In order to deal with the nonlinear part, we will show the existence of fixed point(s) by finding a set of appropriate  convolution operators and applying Schrauder's fixed point theorem.
Defining the set 
\begin{align}
\Gamma_1=\left\{\left(\phi,\psi\right)\in C\left(\R,\R^2\right): \left(0,0\right)\le\left(\phi,\psi\right)\le \left(k_1,k_2\right) \right\}    
\end{align}
we rewrite Eq.(\ref{LV4}) in the form
\begin{equation}\label{fixed point1}
( \phi, \psi)^T = -{\cal L}^{-1}H(\phi,\psi)=\left(-{\cal L}_1^{-1}H_1(\phi,\psi),-{\cal L}_2^{-1}H_2(\phi,\psi)\right)^T , \ \phi,\psi \in \Gamma_1 ,
\end{equation}
Taking $F:=-{\cal L}^{-1}H=-\left(F_1H_1(\phi,\psi),F_2H_2(\phi,\psi)\right)^T.$ Thus, we can use the following Perron-Lyapunov integral operator  \[F=\left(F_1\left(\phi,\psi\right),F_2\left(\phi,\psi\right)\right):\Gamma_1\left(\R,\R^2\right)\to C\left(\R,\R^2\right)\]  defined by 
\begin{align}
F_1\left(\phi,\psi\right)&=\int^\infty_{-\infty} G_1(t-s,r)H_1(\phi(s),\psi(s))ds\\
F_2\left(\phi,\psi\right)&=\int^\infty_{-\infty} G_2(t-s,r)H_2(\phi(s),\psi(s))ds. \nonumber
\end{align}
Here, we have positive constants $M_1, M_2, \delta_1, \delta_2$ that are uniformly bounded in $r$ for some sufficiently small $r>0$  such that for all $t\in \R$
\begin{align*}
&|G_1(t,r)| \le M_1e^{-\delta_1 |t|}\\
&|G_2(t,r)| \le M_2e^{-\delta_2 |t|}.
\end{align*}
See  Theorem  4.1 \cite{mal}.
\begin{lemma} Define $F=\left(F_1\left(\phi,\psi\right),F_2\left(\phi,\psi\right)\right)$ as above, then for any $\left(\phi,\psi\right)\in \Gamma_1,$ then  
\begin{enumerate}
\item$\left(F_1\left(\phi,\psi\right),F_2\left(\phi,\psi\right)\right):\Gamma_1\left(\R,\R^2\right)\to C\left(\R,\R^2\right)$ is well defined.
\item $F_1\left(\phi,\psi\right),F_2\left(\phi,\psi\right)$ satisfy 
\begin{align}
&D_1F_1^{\prime \prime}\left(\phi,\psi\right) -cF_1^{\prime}\left(\phi,\psi\right) -\beta_1 F_1\left(\phi,\psi\right) +H_1\left(\phi,\psi\right)=0\\
&D_2F_2^{\prime \prime}\left(\phi,\psi\right) -cF_2^{\prime}\left(\phi,\psi\right) -\beta_2 F_2\left(\phi,\psi\right) +H_2\left(\phi,\psi\right)=0. \nonumber
\end{align}
\end{enumerate}
\end{lemma}

\medskip
\noindent
The proof is straight forward, and is hence omitted. Moreover, if $\left(\phi,\psi\right)$ is a fixed point of $\left(F_1\left(\phi,\psi\right),F_2\left(\phi,\psi\right)\right),$ then the system (\ref{LV4}) has a traveling wave solution. To this end, $\left(F_1\left(\phi,\psi\right),F_2\left(\phi,\psi\right)\right),$ enjoys similar properties as $\left(H_1\left(\phi,\psi\right),H_2\left(\phi,\psi\right)\right),$ in Lemma (\ref{lemH}). Indeed,
\begin{lemma}\label{lemF}
Assume that $(C1)$ and (PQM) hold, then 
\begin{enumerate}
\item $F_2\left(\phi,\psi\right)(t)$ is non-decreasing for $t\in \R$
\item $F_1(\phi_2,\psi_1)(t)\le F_1(\phi_1,\psi_1)(t)$
\item $F_1(\phi_1,\psi_1)(t)\le F_1(\phi_1,\psi_2)(t)$
\item $F_2\left(\phi_2,\psi_2\right)(t)\le F_2\left(\phi_1,\psi_1\right)(t),$
\end{enumerate}    
when $\phi_1,\ \phi_2, \ \psi_1,\ \psi_2 \in C(\R,\R),\ 0\le \phi_2\le \phi_1\le k_1,\ 0\le \psi_2\le \psi_1\le k_2.$  
\end{lemma}

\medskip
\noindent
The proof follows from the inherited properties of $\left(H_1\left(\phi,\psi\right),H_2\left(\phi,\psi\right)\right)$ from Lemma (\ref{lemH}). 
\subsubsection{Existence Via Schauder's Fixed Point Theorem}
We define the following set 
\begin{align*}
\Gamma_2\left(\left(\underline{\phi}, \underline{\psi}\right),\left(\overline{\phi}, \overline{\psi} \right)\right)=\bigl\{\left(\phi,\psi\right)\in C\left(\R,\R^2\right):{}& i) \ \psi(t) \ \text{is nondecreasing in} \ \R. \ \text{and} \\
&ii)\  \underline{\phi}(t)\le \phi(t)\le \overline{\phi}, \ \underline{\psi}(t)\le \psi(t)\le \overline{\psi}(t) 
\bigr\}.   
\end{align*}
Furthermore, we define the exponential decay norm for some $\mu>0$ as
\[|\Phi|_{\mu}=\sup_{t\in \R} \ e^{-\mu|t|}||\Phi(t)||. \]

\medskip
\noindent
Now, define the ball as
\[B_{\mu}\left(\R,\R^2\right)=\left\{ \Phi\in C\left(\R,\R^2\right): |\Phi|_{\mu}<\infty\right\}.\]
\begin{lemma}\label{LVlem1} Define $\Gamma_2\left(\left(\underline{\phi}, \underline{\psi}\right),\left(\overline{\phi}, \overline{\psi} \right)\right), |\Phi|_{\mu},$ and $B_{\mu}(\R,\R^2)$ as above, then 
\begin{enumerate}
\item $ \Gamma_2$ is nonempty.
\item $\Gamma_2$ is closed, bounded and convex.
\item  $\Gamma_2\subset B_{\mu}(\R,\R^2).$
\item $\left(B_{\mu}(\R,\R^2),|\cdot|_{\mu}\right)$ is a Banach space.
\end{enumerate}   
\end{lemma}
\noindent The result is classical, thus the proof is omitted. It is now possible to complete the requirements of Schauder's theorem. 
\begin{lemma} \label{LVlem2} Assume (C1), (C2) and (PQM) hold, then 
\begin{enumerate}
\item $F:\Gamma_2\left(\left(\underline{\phi}, \underline{\psi}\right),\left(\overline{\phi}, \overline{\psi} \right)\right)\to \Gamma_2 \left(\left(\underline{\phi}, \underline{\psi}\right),\left(\overline{\phi}, \overline{\psi} \right)\right). $
\item $F=(F_1,F_2)$ is continuous with respect to  $|\cdot|_{\mu}$ in $B_{\mu}(\R,\R^2).$
 \end{enumerate}
\end{lemma}
\begin{proof}
For part $i.)$ we only need to show that 
\begin{align*}
\begin{cases}
\underline{\phi}\le F_1(\underline{\phi},\overline{\psi})\le F_1(\overline{\phi},\underline{\psi})   \le\overline{\phi} \\
\underline{\psi}\le  F_2(\overline{\phi},\underline{\psi}) \le F_2(\underline{\phi},\overline{\psi})    \le \overline{\psi}.
\end{cases}  
\end{align*}
This is due to the fact 
\begin{align*}
\begin{cases}
 F_1(\underline{\phi},\overline{\psi})\le F_1(\phi,\psi)\le  F_1(\overline{\phi},\underline{\psi})   \\
  F_2(\overline{\phi},\underline{\psi}) \le F_2(\phi,\psi)\le F_2(\underline{\phi},\overline{\psi}).
\end{cases}  
\end{align*}
We only show the first inequality, because the second follows the same way. We first show that $ \underline{\phi}\le F_1(\underline{\phi},\overline{\psi})$. 
\begin{align*}
&F_1(\underline{\phi},\overline{\psi})(t)=\int_{-\infty}^{\infty} G_1(t-s,r)H_1 (\underline{\phi},\overline{\psi})(s) ds\\
&\ge \int_{-\infty}^{\infty} G_1(t-s,r)\left(-D_1\underline{\phi}''(t)+c\underline{\phi}'(t+r_1)+\beta_1\underline{\phi}(t+r_1)\right) ds=\underline{\phi}(t)
\end{align*}
for all $t\in \R.$
Next, we show $F_1(\overline{\phi},\underline{\psi})   \le\overline{\phi}$.
\begin{align*}
&F_1(\overline{\phi},\underline{\psi}) (t)=\int_{-\infty}^{\infty} G_1(t-s,r)H_1 (\overline{\phi},\underline{\psi})(s) ds\\
&\le \int_{-\infty}^{\infty} G_1(t-s,r)\left(-D_1\overline{\phi}''(t)+c\overline{\phi}'(t+r_1)+\beta_1\overline{\phi}(t+r_1)\right) ds=\overline{\phi}(t)
\end{align*}
for all $t\in \R.$ The proof is complete. The method of proof for part $ii.)$ is similar to  \cite{lilin}, Lemma 3.4 with several modifications. To this end, we only need to show 
\[F_1:B_{\mu}(\R,\R^2)\to B_{\mu}(\R,\R^2) \] is continuous with respect to $|\cdot|_{\mu},$ because the proof for $F_2(\phi,\psi)$ is very similar. Take $\mu<\max\{\delta_1,\delta_2\}$ and let $\Phi=(\phi_1,\psi_1), \Psi=(\phi_2,\psi_2) \in B_{\mu}(\R,\R^2). $ It is clear that \[F_1=(F_1(\Phi), F_1(\Psi)) \in B_{\mu}(\R,\R^2). \] 
We will now turn our attention to the continuity of $F$. Fix $\varepsilon>0,$ and take \[\delta<\min\left\{\frac{\varepsilon(\delta_1-\mu)}{2 M_1},\frac{\varepsilon}{e^{\mu c \tau}L_1+\beta_1}\right\},\]
where $L_1$ is the Lipchitz constant from (C2), $\beta_1$ is from (PQM), and $\delta_1, M_1$ are from the bounds of the Green function. We will first prove that 
$H_1:B_{\mu}(\R,\R^2)\to B_{\mu}(\R,\R^2) $ is continuous with respect to $|\cdot|_{\mu}$. We take $\left|\Phi(t)-\Psi(t)\right|_{\mu}<\delta,$ then 
\begin{align*}
&A:=\left|H_1(\Phi)(t)-H_1(\Psi)(t)\right|_{\mu}\\
&=\left|f_{1c}(\phi_{1t},\psi_{1t})+\beta_1\phi_1(t+r_1)-\left(f_{2c}(\phi_{2t},\psi_{2t})+\beta_1\phi_2(t+r_1)\right) \right|_{\mu}\\
&\le \left|f_{1c}(\phi_{1t},\psi_{1t})-f_{2c}(\phi_{2t},\psi_{2t})\right|_{\mu}+\beta_1\left|\phi_1(t+r_1)-\phi_2(t+r_1) \right|_{\mu}\\
&\le L_1 ||\Phi(t)-\Psi(t)||_{\X_{c\tau}}e^{-\mu|t|} +\beta_1 \sup_{t\in \R}\left|\phi_1(t)-\phi_2(t) \right|_{\mu}\\
&\le L_1 \sup_{\theta\in (-c\tau,0)} |\Phi(t+\theta)-\Psi(t+\theta)|e^{-\mu|t+\theta|}+\beta_1 |\Phi(t)-\Psi(t)|_{\mu}\\
&\le e^{\mu c\tau} L_1  |\Phi(t)-\Psi(t)|_{\mu}+\beta_1 \left|\Phi(t)-\Psi(t) \right|_{\mu}<\left(e^{\mu c\tau} L_1+\beta_1\right)\delta\\
&<\left(e^{\mu c\tau} L_1+\beta_1\right)\left(\frac{\varepsilon}{e^{\mu c \tau}L_1+\beta_1}\right)<\varepsilon.
\end{align*}
Thus, \[H_1:B_{\mu}(\R,\R^2)\to B_{\mu}(\R,\R^2)\] is continuous. The proof for $H_2(\phi,\psi)(t)$ is similar. Therefore, \[H=(H_1,H_2):B_{\mu}(\R,\R^2)\to B_{\mu}(\R,\R^2)\] is continuous with respect to $|\cdot|_{\mu}.$ 
We can now prove that $F_1$ is continuous in the same manner. Indeed,
\begin{align*}
&B:=\left|F_1(\Phi)(t)-F_1(\Psi)(t)\right|=\left|F_1(\phi_{1},\psi_{1})-F_1(\phi_{2},\psi_{2})\right| \\
&=\left|\int_{-\infty}^{\infty} G_1(t-s,r)\left(H_1(\phi_{1},\psi_{1})(s)-H_1(\phi_{2},\psi_{2})(s)\right) ds \right| \\
&\le \int_{-\infty}^{t}\left| G_1(t-s,r)\left(H_1(\phi_{1},\psi_{1})(s)-H_1(\phi_{2},\psi_{2})(s)\right)\right| ds \\
&+\int_{t}^{\infty} \left|G_1(t-s,r)\left(H_1(\phi_{1},\psi_{1})(s)-H_1(\phi_{2},\psi_{2})(s)\right)\right| ds \\
&=\int_{-\infty}^{t}\left| G_1(t-s,r)\right|\left|\left(H_1(\phi_{1},\psi_{1})(s)-H_1(\phi_{2},\psi_{2})(s)\right)\right|ds\\
&+\int_{t}^{\infty}\left| G_1(t-s,r)\right|\left|\left(H_1(\phi_{1},\psi_{1})(s)-H_1(\phi_{2},\psi_{2})(s)\right)\right|ds.
\end{align*}

\noindent
We can now use the fact that $(H_1(\phi,\psi)(t)$ is continuous with respect to the exponential decay norm.
\begin{align*}
&B=\int_{-\infty}^{t}\left| G_1(t-s,r)\right|e^{\mu |s|}\left|\left(H_1(\phi_{1},\psi_{1})(s)-H_1(\phi_{2},\psi_{2})(s)\right)\right|e^{-\mu |s|}ds\\
&+\int_{t}^{\infty}\left| G_1(t-s,r)\right|e^{\mu |s|}\left|\left(H_1(\phi_{1},\psi_{1})(s)-H_1(\phi_{2},\psi_{2})(s)\right)\right|e^{-\mu |s|}ds \\
&\le\left(\int_{-\infty}^{t}\left| G_1(t-s,r)\right|e^{\mu |s|}ds+\int_{t}^{\infty}\left| G_1(t-s,r)\right|e^{\mu |s|}ds\right)\\
&\times \sup_{s\in \R}\left|\left(H_1(\phi_{1},\psi_{1})(s)-H_1(\phi_{2},\psi_{2})(s)\right)\right|_{\mu}. 
\end{align*}
\noindent
 Here, we will use the fact that $G_1(t,r)$ is bounded and decays as $|t|\to \infty$ to see the following:
\begin{align*}
&B\le \delta\left(\int_{-\infty}^{t}\left| G_1(t-s,r)\right|e^{\mu |s|}ds+\int_{t}^{\infty}\left| G_1(t-s,r)\right|e^{\mu |s|}ds\right)\\
&\le  \delta\left(\int_{-\infty}^{t}M_1e^{-\delta_1 |t-s|}e^{\mu|s|}ds+\int_{t}^{\infty}M_1e^{-\delta_1 |t-s|}e^{\mu|s|}ds\right)\\
&\le  M_1\delta\left(\int_{0}^{t}e^{-\delta_1 (t-s)+\mu s}ds+\int_{-\infty}^{0}e^{-\delta_1 (t-s)-\mu s}ds+\int_{t}^{\infty}e^{\delta_1 (t-s)+\mu s}ds\right)\\
&<2M_1\delta\left(\frac{\delta_1 e^{\mu t}+\mu e^{-\delta_1 t}}{\delta_1^2-\mu^2}\right).
\end{align*}
\noindent
In the exponential decay norm for $t>0$ we see 
\begin{align*}
&\left|F_1(\Phi)(t)-F_1(\Psi)(t)\right|_{\mu}=\left|F_1(\Phi)(t)-F_1(\Psi)(t)\right|e^{-\mu t}\\
&\le 2M_1\delta \left(\frac{\delta_1 e^{\mu t}+\mu e^{-\delta_1 t}}{\delta_1^2-\mu^2}\right)e^{-\mu t}=2M_1\delta \left(\frac{\delta_1 +\mu e^{-(\delta_1+\mu)t}}{\delta_1^2-\mu^2}\right)\\
&\le 2M_1\delta \left(\frac{\delta_1 +\mu}{\delta_1^2-\mu^2}\right)= 2M_1\delta \left(\frac{1}{\delta_1-\mu}\right)<\varepsilon
\end{align*}
Moreover, when $t\le 0$ we have the following estimate 
\begin{align*}
&\left|F_1(\Phi)(t)-F_1(\Psi)(t)\right|=\left|F_1(\phi_{1},\psi_{1})-F_1(\phi_{2},\psi_{2})\right| \\ 
&\le\left(\int_{-\infty}^{t}\left| G_1(t-s,r)\right|e^{\mu |s|}ds+\int_{t}^{\infty}\left| G_1(t-s,r)\right|e^{\mu |s|}ds\right)\\
&\times \sup_{s\in \R}\left|\left(H_1(\phi_{1},\psi_{1})(s)-H_1(\phi_{2},\psi_{2})(s)\right)\right|_{\mu} \\
&\le \delta\left(\int_{-\infty}^{t}\left| G_1(t-s,r)\right|e^{\mu |s|}ds+\int_{t}^{\infty}\left| G_1(t-s,r)\right|e^{\mu |s|}ds\right)\\
&\le  M_1\delta\left(\int_{t}^{0}e^{\delta_1 (t-s)-\mu s}ds+\int_{-\infty}^{t}e^{-\delta_1 (t-s)-\mu s}ds+\int_{0}^{\infty}e^{\delta_1 (t-s)+\mu s}ds\right)\\
&<2M_1\delta\left(\frac{\delta e^{-\mu t}+\mu e^{\delta_1 t}}{\delta_1^2-\mu^2}\right).
\end{align*}

\medskip
\noindent
Similarly as above in the exponential norm for $t\le0$ we see 
\begin{align*}
&\left|F_1(\Phi)(t)-F_1(\Psi)(t)\right|_{\mu}=\left|F_1(\Phi)(t)-F_1(\Psi)(t)\right|e^{\mu t}\\
&\le 2M_1\delta \left(\frac{\delta_1 e^{-\mu t}+\mu e^{\delta_1 t}}{\delta_1^2-\mu^2}\right)e^{\mu t}=2M_1\delta \left(\frac{\delta_1 +\mu e^{(\delta_1+\mu)t}}{\delta_1^2-\mu^2}\right)\\
&\le 2M_1\delta \left(\frac{\delta_1 +\mu}{\delta_1^2-\mu^2}\right)= 2M_1\delta \left(\frac{1}{\delta_1-\mu}\right)<\varepsilon.
\end{align*}
Therefore, 
\[F=(F_1,F_2):B_{\mu}(\R,\R^2)\to B_{\mu}(\R,\R^2) \] is continuous with respect to $|\cdot|_{\mu}.$
\end{proof}
\begin{corollary}\label{collip}
Assume $(C2)$ holds, then there is exists a constant $C_1>0$ such that 
\[|F(\Phi)-F(\Psi)|_{\mu}\le C_1|\Phi-\Psi|_{\mu} \ \text{for all} \ \phi, \psi \in \Gamma_2,  \  t\in \R.\]
\end{corollary}
\begin{lemma}\label{LVlem3} Assume $(C1)$ and $(C2)$ hold then
$F(\Gamma_2)\to \Gamma_2 $ is compact.   
\end{lemma}
\begin{proof}
Corollary (\ref{collip}) in conjunction with Lemmas (\ref{LVlem1}) and (\ref{LVlem2})  shows that $F^n(\Gamma_2)$ is equicontinuous and uniformly bounded on any finite interval in $\R$. Thus, take $n\in \N$, then in the interval $[-n,n]$ we can say that $F^n$ is compact by Arzela-Ascoli. Now,  define 
\[F(\Phi,\Psi)(t)=\begin{cases} F(\Phi,\Psi)(t), & \ \text{when}\ t\in [-n,n]\\
F(\Phi,\Psi)(n), & \ \text{when}\ t\in (n,\infty)\\
F(\Phi,\Psi)(-n), & \ \text{when}\ t\in (-\infty,-n).
\end{cases}\]
\\ Fix $t\in \R,$ then for all  $(\phi(t),\psi(t))\in \Gamma_2. $
\begin{align*}
&\sup_{t\in \R} |F^n(\Phi,\Psi)(t)-F(\Phi,\Psi)(t)|_{\mu}\\
&=\sup_{t\in (-\infty,-n)\bigcup (n,\infty)}|F^n(\Phi,\Psi)(t)-F(\Phi,\Psi)(t)|_{\mu}\\
&\le 2C_1 e^{-\mu n}
\end{align*}
Thus, $\lim_{n\to \infty} 2C_1 e^{-\mu n}=0,$ 
 so $F^n\to F$ in $\Gamma_2$ as $n\to \infty.$ 
We can apply Arzela-Ascoli to $F$ as well. The proof is complete.
\end{proof}

\medskip
\noindent
We can now state and prove our main result. 
\begin{definition} A  pair of functions  \
 $\overline{\Phi}=\left(\overline{\phi}, \overline{\psi}  \right), \underline{\Phi}=\left(\underline{\phi}, \underline{\psi}  \right)  \in C^2(\R^2,\R),$ where $ \phi, \phi',\phi'', \psi, \psi',\psi''$ are bounded on $\R$, is called an  upper solution (lower solution, respectively) for the wave equation (\ref{LV3}) if it satisfies the following
\begin{align*}
& D_1\phi''(t)-c\phi'(t+r_1)+f_{1c}(\overline{\phi_t}, \underline{\psi_t})\le 0 , \\
&  D_2\psi''(t)-c\psi'(t+r_3)+f_{2c}(\overline{\phi_t}, \overline{\psi_t})\le 0
\end{align*}
and
\begin{align*}
& D_1\phi''(t)-c\phi'(t+r_1)+f_{1c}(\underline{\phi_t}, \overline{\psi_t})\le 0 , \\
&  D_2\psi''(t)-c\psi'(t+r_3)+f_{2c}(\underline{\phi_t}, \underline{\psi_t})\le 0.
\end{align*}
\end{definition}
\begin{theorem}
Assume $(C1), (C2)$ and $(PQM)$ hold, if there is an upper $(\overline{\phi},\overline{\psi}) \in \Gamma_2$ and a lower solution $(\underline{\phi},\underline{\psi})\in \Gamma_2$ of Eq.(\ref{LV3}) such that for all $t\in \R$
\[0\le  \underline{\psi}(t)\le \overline{\psi}(t), \ 0\le  \underline{\phi}(t)\le \overline{\phi}(t) .\]
Then, there exists a monotone traveling wave solution to the system (\ref{LV3}).
\end{theorem}
\begin{proof}
Lemmas \ref{LVlem1}, \ref{LVlem2}, \ref{LVlem3} allows us to use Schuader's fixed point theorem to show that there exists a fixed point for system(\ref{LV3}). All we need to show is the asymptotic boundary conditions hold. If we apply the asymptotic limits and use the fact that any $(\phi,\psi)\in \Gamma_2$ we have \[(\underline{\phi}(t),\underline{\psi}(t))\le(\phi,\psi)\le(\overline{\phi},\overline{\psi}) \le (k_1,k_2)\] it is easy to see the asymptotic boundary holds. This completes the proof.
\end{proof}

\medskip
\noindent
It is extremely difficult if not impossible to directly find upper/lower solutions directly. It is much easier to relax the conditions and construct upper and/or lower solutions using what are known as quasi-upper/lower solutions, which are "rougher" in nature.
\begin{definition} 
A  pair of functions $\overline{\Phi}=\left(\overline{\phi}, \overline{\psi}  \right), \underline{\Phi}=\left(\underline{\phi}, \underline{\psi}  \right)  \in C^1(\R,\R^2),$ where $ \phi, \phi', \psi, \psi'$  are bounded on $\R$, $\phi'', \psi''$ are locally integrable and essentially bounded on $\R$, is called a quasi- upper solution (quasi-lower solution, respectively) for the wave equation (\ref{LV3}) if it satisfies the following for almost every $t\in \R$
\begin{align*}
& D_1\phi''(t)-c\phi'(t+r_1)+f_{1c}(\overline{\phi_t}, \underline{\psi_t})\le 0 , \\
&  D_2\psi''(t)-c\psi'(t+r_3)+f_{2c}(\overline{\phi_t}, \overline{\psi_t})\le 0
\end{align*}
and
\begin{align*}
& D_1\phi''(t)-c\phi'(t+r_1)+f_{1c}(\underline{\phi_t}, \overline{\psi_t})\le 0 , \\
&  D_2\psi''(t)-c\psi'(t+r_3)+f_{2c}(\underline{\phi_t}, \underline{\psi_t})\le 0.
\end{align*}
\end{definition}
\begin{proposition}
Let $(\phi,\psi)$ be a quasi- upper solution (quasi-lower solution, respectively) of Eq. (\ref{LV3}). Then, $F(\phi,\psi)$ is an upper solution (lower solution, respectively) of Eq. (\ref{LV3}).
\end{proposition}
\begin{proof} This can be done in the same manner for Proposition (4.6) in \cite{BarkNguy}.  Their proof hinges on the construction of an isomorphism between $W^{1,\infty}$ and $L^\infty$, so if $\varphi \in L^{\infty}$ for $T: (\varphi , \varphi')^T\mapsto (0,f).$ This is defined by \cite[Theorem 4.1]{mal}. Also, there exists a unique bounded function $w\in C^2(\R,\R)$ (see e.g. \cite{murnaimin,pru}) such that 
$$
 Dw''(t)-cw'(t+r_1)-\beta w(t+r_1)=-\beta\varphi(t+r_1)-f^c(\varphi_{t+r_1}).
$$  The inequalties can be shown using the monotone property.   
\end{proof}

\begin{corollary}
Let $(\phi,\psi)$ be a quasi- upper solution (quasi-lower solution, respectively) of Eq. (\ref{LV3}). Then, there exists a monotone traveling wave solution of the system (\ref{LV3}).   
\end{corollary}
\section{Applications}
\noindent We will consider the system of equations (\ref{LV1}). That is,
\begin{align}
\frac{\partial u (x,t)}{\partial t}&=D_1\frac{\partial^2u(x,t-\tau_1)}{\partial x^2}+\alpha_1 u(x,t)\left(1-u(x,t)-av(x,t-\tau_2)\right) \\
\frac{\partial v (x,t)}{\partial t}&=D_2\frac{\partial^2v(x,t-\tau_3)}{\partial x^2}+\alpha_2v(x,t)\left(1-v(x,t)-bu(x,t-\tau_4)\right). \nonumber
\end{align}
\begin{lemma} Define
\setlength{\abovedisplayskip}{10pt}
\setlength{\belowdisplayskip}{0pt}
\setlength{\abovedisplayshortskip}{0pt}
\setlength{\belowdisplayshortskip}{0pt}
\begin{eqnarray*}
&f_{1}&(\phi, \psi)=\alpha_1\phi(0)\left(1-\phi(0)-a\psi(-r_2)\right) \\
&f_{2}&(\phi, \psi)= \alpha_2\psi(0)\left(1-\psi(0)-b\phi(-r_4)\right),  
\end{eqnarray*}
\setlength{\abovedisplayskip}{10pt}
\setlength{\belowdisplayskip}{0pt}
\setlength{\abovedisplayshortskip}{0pt}
\setlength{\belowdisplayshortskip}{0pt}
where, $\phi, \psi \in \  C\left([-c\tau, 0],\R^2\right), \ \tau=\max\{\tau_1,\tau_3\}.$
Then $f_{1}(\phi, \psi), f_{2}(\phi, \psi)$ satisfy $ (PQM).$
\end{lemma}
\noindent  For the sake of brevity we only need to show that  $f_{1}(\phi, \psi)$ satisfies the conditions, because $f_{2}(\phi, \psi)$ can be shown in the same manner. To this end take $0\le\phi_2(s)\le\phi_1(s)\le k_1, \ 0\le\psi_2(s)\le\psi_1(s)\le k_2,$ where $\phi_i, \ \psi_i\in C([-c\tau,0],\R).$ For $i.)$ we see by  direct calculation
 
 \begin{align*}
 &f_{1c}\left(\phi_1,\psi_1\right)-f_{1c}\left(\phi_2,\psi_1\right) \\
 &=\alpha_1\phi_1(0)\left(1-\phi_1(0)-a\psi_1(-r_2)\right)-\alpha_1\phi_2(0)\left(1-\phi_2(0)-a\psi_1(-r_2)\right)\\
&=\alpha_1\left[\left(\phi_1(0)-\phi_2(0)\right)-\left(\phi^2_1(0)-\phi^2_2(0)\right)-a\psi_1(-r_2)\right]\\
&=\alpha_1\left(\phi_1(0)-\phi_2(0)\right)\left[1-a\psi_1(-r_2)-\left(\phi_1(0)+\phi_2(0)\right)\right]\\
&\ge \alpha_1\left(\phi_1(0)-\phi_2(0)\right)\left[1-a k_2-2 k_1\right].
 \end{align*}
\noindent Fix $\beta_1>0$ such that $-\beta_1\le \alpha_1(1-a k_2-2 k_1),$ then 
\begin{align*}
 &f_{1c}\left(\phi_1,\psi_1\right)-f_{1c}\left(\phi_2,\psi_1\right)+\beta_1(\phi_1(0)-\phi_2(0))\\
 &\ge  \left(\phi_1(0)-\phi_2(0)\right)\left[\alpha_1\left(1-a k_2-2 k_1\right)+\beta_1\right]\ge 0.
\end{align*}
The proof for part $ii.)$ is also done via direct calculation and rather straight forward. In fact, 
\begin{align*}
&f_{1c}\left(\phi_1,\psi_1\right)-f_{1c}\left(\phi_1,\psi_2\right)=\alpha_1\phi_1(0)\left(1-\phi_1(0)-a\psi_1(-r_2)\right)-\alpha_1\phi_1(0)\left(1-\phi_1(0)-a\psi_2(-r_2)\right)\\
&=\alpha_1 a \left(\psi_2(-r_2)-\psi_1(-r_2)\right)\le 0,
\end{align*}
since $0\le\psi_2(s)\le \psi_1(s)\le k_2$ for all $s\in [-c\tau,0].$
We define traveling wave solutions as $u(x,t)=\phi(x+ct), \ v(x,t)=\psi(x+ct),$ where $ c>0$ is the wave speed. Applying the wave transformation and letting $\xi=x+ct$ gives
\begin{align*}
&u(x,t-\tau_4)=\phi(x+c(t-\tau_4)=  \phi(x+ct-c\tau_4)=\phi(\xi-c\tau_4) \\
&\frac{\partial u(x,t)}{\partial t}=c\phi'(x+ct)=c\phi'(\xi)\\
&\frac{\partial^2 u(x,t-\tau_1)}{\partial x^2}=\phi''(x+c(t-\tau_1))=\phi''(\xi-c\tau_1) 
\end{align*}
For the function $v(x,t)$ we see
\begin{align*}
&v(x,t-\tau_2)=\psi(x+c(t-\tau_2)=  \psi(x+ct-c\tau_2)=\psi(\xi-c\tau_2)\\
&\frac{\partial v(x,t)}{\partial t}=c\psi'(x+ct)=c\psi'(\xi)\\
&\frac{\partial^2 v(x,t-\tau_3)}{\partial x^2}=\psi''(x+c(t-\tau_3))=\psi''(\xi-c\tau_3).
\end{align*}
Furthermore, Taking $r_i=c\tau_i, \ i=1,2,3,4$, $D_1, D_2=1$ and for simplicity we move the delay out of the delay term by taking $t=\xi-r_i, \i=1,3$ the system (\ref{LV1}) becomes
\begin{align}\label{LV2}
&\phi''(t)-c\phi'(t+r_1)+\alpha_1\phi(t+r_1)\left(1-\phi(t+r_1)-a\psi(t+(r_1-r_2))\right) \\
&\psi''(t)-c\psi'(t+r_3)+\alpha_2\psi(t+r_3)\left(1-\psi(t+r_3)-b\phi(t+(r_3-r_4))\right). \nonumber
\end{align}

\medskip
\noindent
The construction of the quasi-upper/lower solutions will be similar to that of the Belousov-Zhabotinskii Equations in the application section of \cite{BarkNguy} and (PQM).  In fact, we can use  Rouch\'e's Theorem to proof the following claim.
 
\begin{claim}\label{claim 1}
Consider, $P_1(\lambda)=\lambda^{2}-c\lambda+\alpha_1=0,$ where 
\[\lambda_{0}=\frac{c+\sqrt{c^{2}-4\alpha_1}}{2}\] 
\medskip
\noindent
is the positive root. Let $c>2\sqrt{\alpha_1}$ and $U_1$ be an open strip $\{ z\in \C\ |\ \lambda_0-\epsilon_1 < \Re z < \lambda_0+\epsilon_1\}$ so that it does not include the other root of $P_1(\lambda)$ in it. Then, for sufficiently small $r_1$ there exists only a single root $\lambda_1(r_1)$ of the equation
\begin{equation}\label{C1}
\lambda^2-c\lambda e^{r_1\lambda}+\alpha_1 e^{r_1\lambda }=0.
\end{equation}
\medskip
\noindent
in $U_1$ that depends continuously on $r_1$. Moreover,  $\lambda_1(r_1)$ is real and
\begin{equation}
\lim_{s\to 0} \lambda_1(r_1) =\lambda_0 .
\end{equation}
\end{claim}
\medskip
\noindent
We also have the following claim by the same reasoning.
\begin{claim}\label{claim 2}
Consider, $P_2(\mu)=\mu^{2}-c\mu+\alpha_2=0,$ where 
\[\mu_{0}=\frac{c+\sqrt{c^{2}-4\alpha_2}}{2}\] is the positive root. Let $c>2\sqrt{\alpha_2}$ and $U_2$ be an open strip $\{ z\in \C\ |\ \mu_0-\epsilon_2 < \Re z < \mu_0+\epsilon_2\}$ so that it does not include the other root of $P_2(\mu)$ in it. Then, for sufficiently small $r_1$ there exists only a single root $\mu_1(r_3)$ of the equation
\begin{equation}\label{C2}
\mu^2-c\mu e^{r_3\mu}+\alpha_2e^{r_3\mu }=0.
\end{equation}
in $U_2$ that depends continuously on $r_3$. Moreover,  $\mu_1(r_3)$ is real and
\begin{equation}
\lim_{s\to 0} \mu_1(r_3) =\mu_0.
\end{equation}
\end{claim}
\medskip
\noindent
We also take $\alpha_2\ge \alpha_1$ such that $0<\mu_0<\lambda_0.$ Thus, $0<\mu_1<\lambda_1.$ Defining $\overline{\phi}_{1}$ and $\overline{\phi}_{2}$ as follows:
\[
\overline{\phi}_{1}(t):=\left\{
\begin{array}
[c]{l}%
\frac{1}{2}e^{\lambda_{1}t},\;\;\;\;\;\;\quad t\leq0,\\
1-\frac{1}{2}e^{-\lambda_{1}t},\quad t>0
\end{array}
\right.  \;\;\;\;\overline{\phi}_{2}(t):=\left\{
\begin{array}
[c]{l}%
\frac{1}{2}e^{\mu_{1}t},\;\;\;\;\;\quad t\leq0,\\
1-\frac{1}{2}e^{-\mu_{1}t},\quad t>0,
\end{array}
\right.
\]
and we set
$$
\underline{\varphi_2}(t)=0, t\in \R ,\quad 
\underline{ \varphi_1}(t)=\begin{cases} \frac{ e^{\lambda_2t}}{4}, \ t< -T ,\\
f(t), \ -T \le  t \le T \\
\frac{1}{2} , \ t> T ,
\end{cases}
$$
where 
\begin{align*}
f(t)&= a(t-T)^3+b(t-T)^2 +\frac{1}{2},
\end{align*}
and $T$ is a large number. The function $f(t)$ enjoys the following properties:
\begin{enumerate}
\item This bridges smoothly the function $e^{\lambda_2t}/4$ and the constant function $1/2$
\item $f(-T)=(1/4)e^{-\lambda_2T}$, $f'(-T)= (\lambda_2/4)e^{-\lambda_2 T}$, $f'(T)=0$, $f(T)=1/2$.
\end{enumerate}
Here, $a$ and $b$ are easily found via simple calculation. In fact,
\begin{align*}
a&=\frac{\lambda_2 Te^{-\lambda_2 T}+e^{-\lambda_2 T}-2}{16T^3}\\
b&=\frac{\lambda_2 Te^{-\lambda_2 T}+6\left( \frac{e^{-\lambda_2 T}}{4}-\frac{1}{2}\right)}{8T^2}.
\end{align*}
We also have the following claim. 
\begin{claim}
 Define $\overline{\phi}_{1},\overline{\phi}_{2},\underline{\phi_1}, \underline{\phi_2}$ as above, then $ \left(\left(\underline{\phi_1}, \underline{\phi_2}\right),\left(\overline{\phi}_{1},\overline{\phi}_{2}\right)\right)\in \Gamma_2.$  
\end{claim}
\begin{lemma}
 For sufficiently small $r_1, r_2$ and $c>\max\{2\sqrt{\alpha_1},2\sqrt{\alpha_2}\}$ where $\overline{\phi}_{1},\overline{\phi}_{2},\underline{\phi_1}, \underline{\phi_2}$ for all $ t\in \R$, then $\left(\overline{\phi}_{1},\overline{\phi}_{2}\right)^T$ is a quasi-upper solution of Eq. \ref{LV2}.  
\end{lemma}
\begin{proof}The proof will be completed in cases. Indeed, fix $r=\max\{r_1,r_3\}.$

\medskip
{\bf Case 1: $t\le - r. $} Direction substitution into the first equation of Eq. \ref{LV2} yields:

\begin{align*}
&\overline{\phi}_{1}''(t)-c\overline{\phi}_{1}'(t+r_1)+\alpha_1\overline{\phi}_{1}(t+r_1)\left(1-\overline{\phi}_{1}(t+r_1)-a\underline{\phi_2}(t+(r_1-r_2)\right)\\
&=\frac{\lambda_{1}^{2}}{2}e^{\lambda_{1}t}-c\frac{\lambda_{1}}{2}e^{\lambda_{1}(t+r_1)}+\frac{\alpha_1}{2}e^{\lambda_{1}(t+r_1)}\left(1-\frac{1}{2}e^{\lambda_{1}(t+r_1-r_2)}\right)\\
&=-\frac{\alpha_1}{2}e^{\lambda_{1}(t+r_1)}\left(\frac{1}{2}e^{\lambda_{1}(t+r_1-r_2)}\right)\le 0.
\end{align*}
For the second equation we have 
\begin{align*}
&\overline{\phi}_{2}''(t)-c\overline{\phi}_{2}'(t+r_3)+\alpha_2\overline{\phi}_{2}(t+r_3)\left(1-\overline{\phi}_{2}(t+r_3)-b\overline{\phi_1}(t+(r_3-r_4))\right)\\
&=\frac{\mu_{1}^{2}}{2}e^{\mu_{1}t}-c\frac{\mu_{1}}{2}e^{\mu_{1}(t+r_3)}+\frac{\alpha_2}{2}e^{\mu_{1}(t+r_3)}\left(1-\frac{1}{2}e^{\mu_{1}(t+r_3)}-\frac{b}{2}e^{\lambda_{1}(t+(r_3-r_4))}\right)\\ 
&=e^{\mu_{1}t}\left(\frac{\mu_{1}^{2}}{2}-c\frac{\mu_{1}}{2}e^{\mu_{1}r_3}+\frac{\alpha_2}{2}e^{\mu_{1}r_3}\right)-\frac{\alpha_2}{2}e^{\mu_{1}(t+r_3)}\left(\frac{1}{2}e^{\mu_{1}(t+r_3)}+\frac{b}{2}e^{\lambda_{1}(t+(r_3-r_4))}\right)\\
&=-\frac{\alpha_2}{2}e^{\mu_{1}(t+r_3)}\left(\frac{1}{2}e^{\mu_{1}(t+r_3)}+\frac{b}{2}e^{\lambda_{1}(t+(r_3-r_4))}\right)\le 0.
\end{align*}
{\bf Case 2: $-r\le t\le 0. $} Direct substitution into the first equation of Eq. \ref{LV2} yields:
\begin{align*}
&\overline{\phi}_{1}''(t)-c\overline{\phi}_{1}'(t+r_1)+\alpha_1\overline{\phi}_{1}(t+r_1)\left(1-\overline{\phi}_{1}(t+r_1)-a\underline{\phi_2}(t+(r_1-r_2))\right)\\
&=\frac{\lambda_{1}^{2}}{2}e^{\lambda_{1}t}-c\frac{\lambda_{1}}{2}e^{-\lambda_{1}(t+r_1)}+\alpha_1\overline{\phi}_{1}(t+r_1)\left(1-\overline{\phi}_{1}(t+r_1)\right)\\
& =\frac{\lambda_{1}^{2}}{2}e^{\lambda_{1}t}-c\frac{\lambda_{1}}{2}e^{\lambda_{1}(t+r_1)}+\frac{\alpha_{1}}{2}e^{\lambda_{1}(t+r_1)}+c\frac{\lambda_{1}}{2}e^{\lambda_{1}(t+r_1)}-\frac{\alpha_{1}}{2}e^{\lambda_{1}(t+r_1)}-c\frac{\lambda_{1}}{2}e^{-\lambda_{1}(t+r_1)}\\
&+\alpha_1\overline{\phi}_{1}(t+r_1)\left(1-\overline{\phi}_{1}(t+r_1)\right)\\
&=c\frac{\lambda_{1}}{2}e^{\lambda_{1}(t+r_1)}-\frac{\alpha_{1}}{2}e^{\lambda_{1}(t+r_1)}-c\frac{\lambda_{1}}{2}e^{-\lambda_{1}(t+r_1)}+\alpha_1\left(1-\frac{1}{2}e^{-\lambda_{1}(t+r_1)}\right)\frac{1}{2}e^{-\lambda_{1}(t+r_1)}\\
&=(c\lambda_1-\alpha_1)\sinh(\lambda_{1}(t+r_1))-\frac{\alpha_1}{4}e^{-2\lambda_{1}(t+r_1)}.
\end{align*}
Using the power expansion of $\sinh x$ we have
$
\sinh (\lambda_1(t+r_1)) \approx \lambda_1(t+r_1 )+o(r^2_1).$ Thus, for some small $r_1$ we have \[(c\lambda_1-\alpha_1)\sinh(\lambda_{1}(t+r_1))-\frac{\alpha_1}{4}e^{-2\lambda_{1}(t+r_1)}=(c\lambda_1-\alpha_1)\left(\lambda_1(t+r_1 )+o(r^2_1)\right)-\frac{\alpha_1}{4}e^{-2\lambda_{1}(t+r_1)}\le 0.\]
For the second equation we have 
\begin{align*}
&\overline{\phi}_{2}''(t)-c\overline{\phi}_{2}'(t+r_3)+\alpha_2\overline{\phi}_{2}(t+r_3)\left(1-\overline{\phi}_{2}(t+r_3)-b\overline{\phi_1}(t+(r_3-r_4))\right)\\
&=\frac{\mu_1^2}{2}e^{\mu_1 t}-\frac{c\mu_1}{2}e^{-\mu_1(t+r_3)}+\alpha_2\left(1-\frac{1}{2}e^{-\mu_1(t+r_3)}\right)\left[1-\left(1-\frac{1}{2}e^{-\mu_1(t+r_3)}\right)-\frac{b}{2}e^{\lambda_1(t+r_3-r_4)}\right]\\
&=\frac{\mu_1^2}{2}e^{\mu_1 t}-\frac{c\mu_1}{2}e^{\mu_1(t+r_3)}+\frac{\alpha_2}{2}e^{\mu_1(t+r_3)}+\frac{c\mu_1}{2}e^{\mu_1(t+r_3)}-\frac{\alpha_2}{2}e^{\mu_1(t+r_3)}-\frac{c\mu_1}{2}e^{-\mu_1(t+r_3)}\\
&+\alpha_2\left(1-\frac{1}{2}e^{-\mu_1(t+r_3)}\right)\left[1-\left(1-\frac{1}{2}e^{-\mu_1(t+r_3)}\right)-\frac{b}{2}e^{\lambda_1(t+r_3-r_4)}\right]\\
&=c\mu_1\sinh(\mu_1r_3)-\frac{\alpha_2}{2}\sinh(\mu_1r_3)-\alpha_2\left[\frac{1}{4}e^{-2\mu_1(t+r_3)}+\frac{b}{2}e^{\lambda_1(t+r_3-r_4)}\left(1-\frac{1}{2}e^{-\mu_1(t+r_3)}\right)\right].
\end{align*}
\noindent This is less than zero as $r_3\to 0$, so we have the result.
\medskip 

{\bf Case 3: $0< t. $} Direction substitution into the first equation of Eq. \ref{LV2} yields:
\begin{align*}
&\overline{\phi}_{1}''(t)-c\overline{\phi}_{1}'(t+r_1)+\alpha_1\overline{\phi}_{1}(t+r_1)\left(1-\overline{\phi}_{1}(t+r_1)-a\underline{\phi_2}(t+(r_1-r_2))\right)\\
&=\frac{-\lambda_{1}^{2}}{2}e^{-\lambda_{1}t}-\frac{c\lambda_{1}}{2}e^{-\lambda_{1}(t+r_1)}+\alpha_1\left(1-\frac{1}{2}e^{-\lambda_{1}(t+r_1)}\right)\left(1-\left(1-\frac{1}{2}e^{-\lambda_{1}(t+r_1)}\right)\right)\\
&=\frac{-\lambda_{1}^{2}}{2}e^{-\lambda_{1}t}+\frac{c\lambda_{1}}{2}e^{-\lambda_{1}t+\lambda_{1}r_1}-\frac{\alpha_{1}}{2}e^{-\lambda_{1}t+\lambda_{1}r_1}-\frac{c\lambda_{1}}{2}e^{-\lambda_{1}t+\lambda_{1}r_1}+\frac{\alpha_{1}}{2}e^{-\lambda_{1}t+\lambda_{1}r_1}\\
&-\frac{c\lambda_{1}}{2}e^{-\lambda_{1}(t+r_1)}+\alpha_1\left(1-\frac{1}{2}e^{-\lambda_{1}(t+r_1)}\right)\left(1-\left(1-\frac{1}{2}e^{-\lambda_{1}(t+r_1)}\right)\right)\\
\end{align*}
It is clear that \[-\alpha_1\left(1-\frac{ e^{-\lambda_1t-\lambda_1 r_1}}{2}\right)\left(1-\frac{ e^{-\lambda_1 (t+r_1-r_2)}}{2}\right)\le 0.\]
Furthermore, \[\lim_{r_1\to 0} \alpha_1-c\lambda_1\cosh(\lambda_1 r_1)+\alpha_1\sinh(\lambda_1 r_1)=\alpha_1-c\lambda_0, \]
because $\lambda_1\to \lambda_0.$ A simple calculation shows
\begin{align*}
 \alpha_1-c\lambda_0=\alpha_1-c\left(\frac{c+\sqrt{c^{2}-4\alpha_1}}{2}\right)\le 0,  
\end{align*}
due to the fact that $c\ge 2\sqrt{\alpha_1}.$ This means that there exists some $r^*>0$ dependent on $c$ such that $0<r_1\le r^*(c)$ and \[\left(\alpha_1-c\lambda_1\cosh(\lambda_1 r_1)+\alpha_1\sinh(\lambda_1 r_1)\right)e^{-\lambda_1 t}-\alpha_1\left(1-\frac{ e^{-\lambda_1t-\lambda_1 r_1}}{2}\right)\left(1-\frac{ e^{-\lambda_1 (t+r_1)}}{2}\right)\le 0.\]
For the second equation we have 
\begin{align*}
&\overline{\phi}_{2}''(t)-c\overline{\phi}_{2}'(t+r_3)+\alpha_2\overline{\phi}_{2}(t+r_3)\left(1-\overline{\phi}_{2}(t+r_3)-b\overline{\phi_1}(t+r_3-r_4)\right)\\
&=\frac{-\mu_{1}^{2}}{2}e^{-\mu_{1}t}-c\frac{\mu_{1}}{2}e^{-\mu_{1}(t+r_3)}+\alpha_2\left(1-\frac{1}{2}e^{-\mu_{1}(t+r_3)}\right)\\
&\times\left(1-\left(1-\frac{1}{2}e^{-\mu_{1}(t+r_3)}\right)-b\left(1-\frac{1}{2}e^{-\lambda_{1}(t+r_3-r_4)}\right)\right)\\
&=\left(\alpha_2-c\mu_1\cosh(\mu_1 r_3)+\alpha_2\sinh(\mu_1 r_3)\right)e^{-\mu_1 t}\\
&-\alpha_2\left(\left(1-\frac{1}{2}e^{-\mu_{1}(t+r_3)}\right)+b\left(1-\frac{1}{2}e^{-\lambda_{1}(t+r_3-r_4)}\right)\right)\le 0
\end{align*}
by the same reasoning as the first equation in this case. This proves the Lemma.
\end{proof}
\begin{lemma}
 For sufficiently small $r_1, r_2$ and $c>\max\{2\sqrt{\alpha_1},2\sqrt{\alpha_2}\}$ where $\overline{\phi}_{1},\overline{\phi}_{2},\underline{\phi_1}, \underline{\phi_2}$ for all $ t\in \R$, then $\left(\underline{\phi_1}, \underline{\phi_2}\right)^T$ is a quasi-lower solution of Eq. \ref{LV2}.  
\end{lemma}
\begin{proof} We will show this Lemma in cases like above. In fact, we only need to show that the inequality is satisfied for the first equation in the model due to the fact for all $t\in \R$ the second equation always is zero. This is shown by
\begin{align*}
&\underline{\phi}_{2}''(t)-c\underline{\phi}_{2}'(t+r_3)+\alpha_2\underline{\phi}_{2}(t+r_3)\left(1-\underline{\phi}_{2}(t+r_3)-b\underline{\phi_1}(t+r_3-r_4)\right)=0.
\end{align*} 

\noindent
\medskip
{\bf Case 1: $t\le -T- r. $} Direction substitution into the first equation of Eq. \ref{LV2} yields:
\begin{align*}
 &\underline{\phi}_{1}''(t)-c\underline{\phi}_{1}'(t+r_1)+\alpha_1\underline{\phi}_{1}(t+r_1)\left(1-\underline{\phi}_{1}(t+r_1)-a\overline{\phi_2}(t+r_1-r_2)\right)\\
&=  \left[ \underline{\phi}_{1}^{\prime\prime}(t)-c\underline{\phi}_{1}^{\prime}(t+r_1)+\frac{\alpha_1}{2}  \underline{\phi}_{1}(t+r_1)\right]
+ \alpha_1\underline{\phi}_{1}(t+r_1)\left(
\frac{1}{2}-\underline{\phi}_{1}(t+r_1)-a\overline{\phi_2}(t+r_1-r_2)\right) \\
&=e^{\lambda_2 t}\left[\lambda_2^2-c\lambda_2e^{\lambda_2 r_1} +\frac{\alpha_1 }{2}e^{\lambda_2 r_1} \right]
+ \alpha_1\underline{\phi}_{1}(t+r_1)\left(
\frac{1}{2}-\underline{\phi}_{1}(t+r_1)-a\overline{\phi_2}(t+r_1-r_2)\right)\\
&=\alpha_1\underline{\phi}_{1}(t+r_1)\left(
\frac{1}{2}-\underline{\phi}_{1}(t+r_1)-a\overline{\phi_2}(t+r_1-r_2)\right).
\end{align*}
 However, we can take $T$ to be sufficiently large enough such that \[\alpha_1\underline{\phi}_{1}(t+r_1)\left(
\frac{1}{2}-\underline{\phi}_{1}(t+r_1)-a\overline{\phi_2}(t+r_1-r_2)\right)\ge 0.\]
\noindent
\medskip
{\bf Case 2: $-r-T< t\le T . $} Direction substitution into the first equation of Eq. \ref{LV2} yields:
\begin{align*}
 &\underline{\phi}_{1}''(t)-c\underline{\phi}_{1}'(t+r_1)+\alpha_1\underline{\phi}_{1}(t+r_1)\left(1-\underline{\phi}_{1}(t+r_1)-a\overline{\phi_2}(t+r_1-r_2)\right)\\
&= \frac{1}{4}\left( f''(t)-cf'(t+r_1) +\alpha_1 f(t+r_1)\left(1-\frac{f(t+r_1)}{4}-a\overline{\phi_2}(t+r_1-r_2)\right)\right).
\end{align*}
Again, we can take $T$ large enough such that \[\alpha_1 f(t+r_1)\left(1-\frac{f(t+r_1)}{4}-a\overline{\phi_2}(t+r_1-r_2)\right)\ge 0.\]
In \cite{BarkNguy} it was noted that on this interval
\begin{align*}
&= \sup_{-T-r_1 \le t\le T}  |\underline{\phi}_{1}^{\prime\prime}(t)|+c|\underline{\phi}_{1}^{\prime}(t+r_1)|\\
&= \sup_{-T\le t\le T} |f'(t)|+c|f''(t)|,
\end{align*}
so we have that 
\[\frac{1}{4}\left( f''(t)-cf'(t+r_1) +\alpha_1 f(t+r_1)\left(1-\frac{f(t+r_1)}{4}-a\overline{\phi_2}(t+r_1-r_2)\right)\right)\ge 0.\]

\noindent
\medskip
{\bf Case 3: $ t \ge T . $} Direction substitution into the first equation of Eq. \ref{LV2} yields:
\begin{align*}
 &\underline{\phi}_{1}''(t)-c\underline{\phi}_{1}'(t+r_1)+\alpha_1\underline{\phi}_{1}(t+r_1)\left(1-\underline{\phi}_{1}(t+r_1)-a\overline{\phi_2}(t+r_1-r_2)\right)\\
&=\frac{\alpha_1}{2}\left(1-\frac{1}{2}-a\overline{\phi_2}(t+r_1-r_2)\right)\ge 0.
\end{align*}
This concludes the proof.
\end{proof}
\begin{corollary}
Assume that $c>\max\{2\sqrt{\alpha_1},2\sqrt{\alpha_2}\}$ is given. Then, the system (\ref{LV2}) has a traveling wave solution $(u(x,t),v(x,t))^T=\Phi=\left({\phi(x+ct)}, {\psi(x+ct)}  \right)$ for sufficiently small delays $\tau_1,\tau_2$.
\end{corollary}

\newpage
\bibliographystyle{amsplain}

\end{document}